\documentclass{amsart}
\usepackage{amssymb}
\usepackage{amsmath}
\usepackage{color}
\usepackage{graphicx}
\usepackage{dsfont}

\newtheorem{theo}{\indent Theorem}[section]
\newtheorem{prop}[theo]{\indent Proposition}

\newtheorem{rem}[theo]{\indent Remark}
\newtheorem{lem}[theo]{\indent Lemma}
\newtheorem{defin}[theo]{\indent Definition}

\newtheorem{ass}[theo]{\indent Assumption}

\def\D{{\mathbb D}}
\def\cP{{\mathcal P}}
\def\vip{\vskip0.2cm}

\def \E{I\!\!E}
\def \P{I\!\!P}

\def \intot{\int_0^t }
\def \cS{\mathcal S}
\def \cB{\mathcal B}
\def \cL{\mathcal L}
\def \cP{\mathcal P}

\newcommand{\R}{\mathbb {R}}
\newcommand{\N}{\mathbb {N}}

\newcommand{\bN}{{\mathbf{N}}}
\def\indiq{{\bf 1}}

\newlength{\breite}
\title[Hawkes with inhibition]{Stability for Hawkes processes with inhibition}
\date{}
\author{Mads Bonde Raad}
\author{Eva L{\"o}cherbach}

\address{Mads Bonde Raad: Department of Mathematical Sciences,
		University of Copenhagen, Universitetsparken 5, 2100 Copenhagen, Danemark.}

\email{madsraad@gmail.com}

\address{E. L\"ocherbach: Universit{\'e} de Paris 1 Panth\'eon-Sorbonne, SAMM,  EA 4543 et FR FP2M 2036 CNRS, 75013 Paris, France.}

\email{eva.locherbach@univ-paris1.fr}

\begin{document}
\maketitle
\def\abstractname{Abstract}
\begin{abstract}
We consider a multivariate non-linear Hawkes process in a multi-class setup where particles are organised within two populations of possibly different sizes, such that one of the populations acts excitatory on the system while the other population acts inhibitory on the system. The goal of this note is to present a class of Hawkes Processes with stable dynamics without assumptions on the spectral radius of the associated weight function matrix. This illustrates how inhibition in a Hawkes system significantly affects the stability properties of the system. 
\end{abstract}

\\


\section{Introduction and main result}\label{sec:1}
We consider a system of interacting Hawkes processes structured within two populations. We shall label the two populations with ``$+$'' or ``$-$'' signaling that the population acts excitatory or inhibitory on the system, respectively. Let $N_+,N_{-} \in \N$ be the number of units in each population. Introduce weight functions given by 
\begin{eqnarray}\label{eq:weightfunctions}
h_{++} ( t)& =&\frac{c_{++}}{N_{+}} e^{ -\nu_{+}t}, \quad  h_{+-}( t)=\frac{c_{+-}}{N_{+}}e^{ -\nu_{+}t},\\
h_{-+}( t)&=&\frac{c_{-+}}{N_{-}} e^{ -\nu_{-}t},\quad   h_{--}( t)=\frac{c_{--}}{N_{-}} e^{ -\nu_{-}t}, 
\end{eqnarray}
for $ t \geq 0.$ 
In the above formula, $h_{+-}$ indicates the weight function from a unit in the excitatory group ``$+$'' to a unit in the inhibitory group ``$-$'', and so on.  The coefficients of the system of interacting Hawkes processes are the exponential leakage terms $ \nu_+ > 0 , \nu_- > 0  $  and the weights $ c_{++} , c_{+-} , c_{-+}, c_{--} $ satisfying that
\begin{equation}
 c_{++} \geq 0, \; c_{+-} \geq 0, \; c_{--} \leq 0, \;  c_{-+} \leq 0.
\end{equation}
The multivariate linear Hawkes process with these parameters is given as
\begin{eqnarray}\label{eq:dyn}
Z_{+}^i (t)  &= &     \int_0^t  \int_0^\infty
 \indiq_{ \{ z \leq  \psi^{i}_+   ( X_{+} ({s-})) \}} \pi_{+}^i  (ds,  dz) , 1 \le  i\leq N_{+}, \\
Z_{-}^j  (t)  &= &     \int_0^t  \int_0^\infty
 \indiq_{ \{ z \leq  \psi^{j}_-   ( X_{-} ({s-})) \}} \pi_{-}^j (ds,  dz) ,1 \le j\leq N_{-}, \\
X_+ (t) &= & e^{- \nu_+ t } X_+ (0) + \frac{c_{++}}{N_+}  \sum_{i=1}^{N_{+}}\int_0^t e^{- \nu_+ ( t- s) } Z_{+}^i (d s)   + \frac{c_{-+}}{N_-}  \sum_{j=1}^{N_{-}}\int_0^t e^{- \nu_+ ( t- s) } Z_{-}^j  (ds) ,\\
X_{-} (t) &= & e^{- \nu_- t } X_{-} (0 )+ \frac{c_{+-}}{N_+} \sum_{i=1}^{N_{+}}\int_0^t e^{- \nu_- ( t- s) }  Z_{+}^i (ds)   +\frac{ c_{--}}{N_-}  \sum_{j=1}^{N_{-}}\int_0^t e^{- \nu_- ( t- s) }  Z_{-}^j  (ds) ,
\end{eqnarray}
where the  jump rate functions $\psi^i_+  : \R \to \R_+ , \psi^i_- : \R \to \R_+ $ are given by
\begin{equation}\label{eq:aplus}
\psi^i_{\pm} (x) = a^i_{\pm} + \max (x, 0),  \; \mbox{ where } a^i_{\pm } > 0 ,
\end{equation}
and where the $ \pi^i_{\pm}  , i \geq 1, $ are i.i.d. Poisson random measures on $ \R_+ \times \R_+ $ having intensity $ dt dz.$

Notice that the process $ ( X_+, X_{-} ) $ is a piecewise deterministic Markov process having generator 
\begin{multline}
A g (x, y ) = - \nu_+ x \partial_x g (x,y ) - \nu_- y \partial_y g (x, y ) + \sum_{i=1}^{N_+} \psi^i_+ ( x) [ g ( x + \frac{c_{++}}{N_+} , y+ \frac{c_{+-}}{N_+}  )- g(x,y ) ] \\
+\sum_{j=1}^{N_-}  \psi_-^j  (y ) [ g ( x + \frac{c_{-+}}{N_-} , y + \frac{c_{--}}{N_-}  ) - g(x, y ) ] ,
\end{multline}
for sufficiently smooth test functions $g.$ 

Classical stability results for multivariate nonlinear Hawkes processes found e.g. in \cite{bm} or in the recent paper  \cite{manonetal}, which is devoted to the study of the stabilising effect of inhibitions, are stated in terms of an associated  weight function matrix $\Lambda,$ imposing that the spectral radius of $\Lambda$ is strictly smaller than one. In this case the process is termed to be {\it subcritical}. This spectral radius stability condition has a natural  interpretation in terms of a multitype branching process with immigration which is spatially structured and where each jump of a given type ($+ $ or $-$) gives rise to future jumps of the same or of the opposite type, see \cite{ho}. The subcriticality condition ensures the recurrence of this process (see \cite{kaplan}). In our system, the weight function matrix is given by
\begin{equation}\label{eq:Lambda}
\Lambda = \left( \begin{array}{cc}
\frac{c_{++}}{\nu_+} &  \frac{|c_{-+}|}{\nu_+} \\ 
\frac{c_{+-}}{\nu_-} &  \frac{|c_{--}|}{\nu_-}
\end{array}
\right)  .
\end{equation}
Notice that in \eqref{eq:Lambda}, negative synaptic weights do only appear through their absolute values. This is due to the fact that using the Lipschitz continuity of the rate functions leads automatically to considering absolute values and does not enable us to make profit from the inhibitory action of $c_{-+} $ and $ c_{--}. $ Obviously, having sufficiently fast decay, that is,  $ \min (\nu_+ , \nu_-) >> 1, $ is a sufficient condition fo subcriticality. 

The purpose of this note is to show how the presence of sufficiently high (in absolute value) negative  weights helps stabilising the process without imposing such a subcriticality condition, in particular, without imposing $ \nu_+, \nu_- $ being large. To the best of our knowledge, only few results have been obtained  on this natural  question in the literature. \cite{bm} gives an attempt in this direction but does only deal with the case when $ c_{+- } $ and $ c_{-+}$ are of the same sign (see Theorem 6 in \cite{bm}), and \cite{manonetal} do only work with the positive part of the weight functions, without  profiting from the explicit inhibitory part within the system. 

Our approach is based on the construction of a convenient Lyapunov function using the inhibitory part of the dynamics. As such,  this approach is limited to the present Markovian framework where the weight functions are decreasing exponentials. 

In the following, we shall write
$$ c_{++}^* := c_{++} - \nu_+ , \; c_{-- }^* := c_{--} - \nu_{-} .$$
Notice that $ c_{++}^* $ could be interpreted as the net increase of $ X_+ $ due to self-interactions of $X_+ $ with itself. $ c_{--}^* $ is always negative.

\begin{ass}\label{aslol}
We assume the following inequalities.
\begin{eqnarray}\label{eq:stab}
 c_{++}^* +c_{--}^* &<& 0    ,\\
 ( c_{++}^*- c_{--}^* )^2 &< &4 c_{+- } | c_{-+}| ,\\
 c_{++}^* -  c_{--}^* &>& 0.
\end{eqnarray}
\end{ass}

This assumption ensures that the system is balanced. Notice that Assumption \ref{aslol} does not imply - nor is implied by - that the spectral radius of $\Lambda$ is strictly smaller than $1$. For example, if Assumption \ref{aslol} is satisfied for some parameters $ ( c_{++},c_{+-},c_{-+},c_{--},\nu,\nu ) , $ i.e., $\nu_+ = \nu_- = \nu, $  such that additionally $ c_{++} + c_{--} < 0, $ then for all $C>1 $ and all  $\varepsilon>0,$ the set of parameters   $(  Cc_{++},Cc_{+-},Cc_{-+},Cc_{--},\varepsilon\nu,\varepsilon\nu ) $ satisfies Assumption \ref{aslol} as well. But the associated offspring matrix $\Lambda_{C,\varepsilon}$ of the scaled parameters is equal to $(C/\varepsilon) \Lambda , $ and thus the spectral radius is also scaled by $C/\varepsilon$.

\begin{ass}\label{ass:2d}
We assume that either $ \nu_+ \neq \nu_- $ or $ \nu_+ = \nu_- $ and $( c_{++},c_{+-} ) ,( c_{-+},c_{--}) $ are linearly independent.
\end{ass}

We are now able to state our main result. It states that under Assumptions \ref{aslol} and \ref{ass:2d}, the process $ X = (X_+, X_{-} ) $ is positive Harris recurrent, together with a strong mixing result. To state our result, for any $ t > 0  $ and for $ z = (x,y ) \in \R^2 ,$ we write $ P_t ( z, \cdot )$ for the transition semigroup of the process, defined through $P_t (z, A) = E_z ( 1_A (X (t)) ) .$ Moreover, for any pair of probability measures $\mu_1, \mu_2 $ on $ {\mathcal B} (\R^2)$ and for any function $ V : \R^2 \to [1, \infty [, $ we put 
$$  \| \mu_1- \mu_2  \|_{ V} := \sup_{ g : |g| \le  V } | \mu_1 ( g) - \mu_2 (g)  | .$$

\begin{theo}\label{theo:harris}
Grant Assumptions \ref{aslol} and \ref{ass:2d}. \\
1) Then the process $ X = (X_+, X_- ) $ is positive recurrent in the sense of Harris, and its unique invariant probability measure $ \mu $ possesses a Lebesgue continuous part. \\ 
2) There exists a function  $V (x, y ) : \R^2 \to [1, \infty [  $ such that  $\lim_{ |x| + |y| \to \infty } V ( x,y ) = \infty $ and there exist  $ c_1, c_2 > 0 $ such that for all $z \in \R^2$ and all $ t \geq 0, $ 
\begin{equation}\label{eq:last}
\| P_t(z , \cdot )  - \mu\|_{ V} \le c_1  V (z) e^{ - c_2 t} .
\end{equation}  
\end{theo} 

\begin{rem}
Notice that if Assumption \ref{ass:2d} is not satisfied, that is, if $ \nu_+ = \nu_- $ and if
$$\left( \begin{array}{c} c_{-+}\\
 c_{--} 
\end{array}\right) \in H:= \R \left( \begin{array}{c} c_{++}\\
 c_{+-} 
\end{array}\right),$$  
then it is easily shown that almost surely, $ dist ( X (t) , H) \to 0 $ as $t \to \infty $ and that $H$ is invariant under the dynamics. Moreover, the restriction of the dynamics to $H$ is Harris recurrent, having a unique invariant measure  $ \mu $ which is absolutely continuous with respect to the Lebesgue measure on $H.$ However, it is easy to show that the original process $X,$ defined on $ \R^2, $ is not Harris in this case, since it is not $ \mu-$irreducible. 
\end{rem}

\section{Proof of Theorem \ref{theo:harris}}
This section is devoted to the proof of Theorem \ref{theo:harris}. 

\subsection{A Lyapunov function for $X$}
We start this section with the following useful property.
\begin{prop}\label{prop:Feller}
The process $ X$ is a Feller process, that is, for any $f : \R^2 \to \R$ which is bounded and continuous, we have that $\R^2 \ni  (x, y ) = z \mapsto  E_{z} f (X (t) ) = P_t f (z)  $ is continuous. 
\end{prop}

The proof of this result follows from classical arguments, see e.g.\ the proof of Proposition 4.8 in \cite{evaflow}, or \cite{ikeda1966}. 

The next result shows that if the cross-interactions, that is, influence from $ X_+ $ to $X_{-} $ and vice versa, are sufficiently strong, then -- under mild additional assumptions --  it is possible to construct a Lyapunov function for the system that does mainly profit from the inhibitory part of the jumps. 

\begin{prop}\label{prop:lyapunov}
Grant Assumption \ref{aslol} and put 
$$ V ( x, y ) := \left\{ 
\begin{array}{ll}
V_{++} ( x, y ) : =   c_{+- } x^2 -c_{-+}y^2 - (c_{++}^* - c_{--}^*) xy & x \in \R_+ , y\in \R_+ \\
V_{+-} (x,y ) := c_{+- } x^2 + q y^2 - (c_{++}^* - c_{--}^*) xy & x\in \R_+ , y \in \R_- \\
V_{-+} (x,y ) :=   px^2 -c_{-+}y^2 - (c_{++}^* - c_{--}^*) xy & x \in \R_- , y\in \R_+ \\
V_{--} (x,y ) :=  p x^2 + qy^2 - (c_{++}^* - c_{--}^*) xy & x \in \R_- , y\in \R_-
\end{array}
\right\} , $$
with $p$ so small such that 
$$ -  (c_{++}^* -  c_{--}^* ) (c_{--} - \nu_+ - \nu_- ) + 2 p c_{-+} > 0 $$
and
$q$ so large such that 
$$ (c_{++}^* -  c_{--}^* ) [ \nu_+ + \nu_- - c_{++} ] + 2 q c_{+- } > 0 \mbox{ and } 4 pq >  (c_{++}^* - c_{--}^*)^2 .$$
Then $\lim_{ |x| + |y| \to \infty } V ( x,y ) = \infty $ and there exist $ \kappa, c, K  > 0 $ such that 
\begin{equation}
A V (x,y ) \le - \kappa V ( x, y ) + c 1_{\{ | x| + |y| \geq K\}} .
\end{equation}
\end{prop}

\begin{proof}
We calculate $ A V ( x, y ) = A^1 V (x, y ) + A^2 V (x, y ) , $ with 
$$ A^1 V ( x, y ) = - \nu_+ \partial_x V (x,y ) - \nu_- \partial_y V (x, y ) $$ 
and $ A^2 $ the jump part of the generator. 

{\bf Part 1.1} Suppose first that $ x \geq |c_{-+}|/ N_-  , y \geq | c_{--} |/N_-  . $ Then 
$$ A V (x, y ) = A^1 V_{++} (x, y) + A^2 V_{++} (x,y ) = a_{++} x^2 + b_{++} xy + d_{++} y^2 + L_{++} (x,y ) ,$$
where $L_{++} $ is a polynomial of degree $1.$ A straightforward calculus shows that 
\begin{eqnarray*}
a_{++} &=& c_{+-} (c_{++}^* + c_{--}^* )   , \\
b_{++} &=& -  (c_{++}^* -  c_{--}^* ) (c_{++}^* + c_{--}^* )\\
d_{++} & =& - c_{-+} (c_{++}^* + c_{--}^* ) ,
\end{eqnarray*}  
proving that 
$$ A V( x, y ) = (c_{++}^* + c_{--}^* ) V (x,y) + L_{++} (x,y ) .$$
This implies that there exist $K,  \kappa > 0 $ such that 
$$ A V ( x, y ) \le - \kappa V ( x,y ) $$
for all $ x > K , y > K,$ since $ c_{++}^* + c_{--}^*  < 0 $ by assumption. 

{\bf Part 1.2} Suppose now that $0 \le  x < |c_{-+} |/N_-  $ and $y \geq | c_{--} |/N_- .$  Then a jump of one of the inhibitory neurons will lead to a change $ x \mapsto x + c_{-+}/N_-  < 0 .$ In this case we obtain
$$ A V (x,y) = A V_{++} ( x,y ) + \sum_{j=1}^{N_-}(a^{j}_-  + y ) (  V_{-+} ( x + \frac{c_{-+}}{N_-}, y + \frac{c_{--}}{N_-} ) - V_{++} ( x + \frac{c_{-+}}{N_-}, y + \frac{c_{-+}}{N_-} )) .$$
But 
$$ |V_{-+} ( x + \frac{c_{-+}}{N_-}, y + \frac{c_{--}}{N_-} ) - V_{++} ( x + \frac{c_{-+}}{N_-}, y + \frac{c_{-+}}{N_-} )| \le C ,$$
since $ | x| < |c_{-+} |,$ and therefore
$$ A V (x,y) \le  A V_{++} ( x,y ) + L (y) , $$
where $ L(y) $ is a monomial in $ y.$ 

The other case $0 \le  y < |c_{--} |/N_-  $ and $x \geq | c_{-+} |/N_- $ is treated analogously.  

{\bf Part 2.1} Suppose now that $ x \geq |c_{-+}|/N_-  , y \leq   - c_{+-} /N_+  . $
Then
$$ A V (x, y ) = A^1 V_{+-} (x, y) + A^2 V_{+-} (x,y ) = a_{+-} x^2 + b_{+-} xy + d_{+-} y^2 + L_{+-} (x,y ) ,$$
where $L_{+-} $ is a polynomial of degree $1.$ We obtain
\begin{eqnarray*}
a_{+-} &=& c_{+-} (c_{++}^* + c_{--}^* )   , \\
b_{+-} &=&   (c_{++}^* -  c_{--}^* ) (\nu_+ + \nu_- - c_{++} ) + 2 q c_{+- } \\
d_{+- } & =& - 2 \nu_- q  .
\end{eqnarray*} 
Since $ b_{+-} > 0 $ by choice of $q,$ this implies that for a suitable positive constant $ \kappa > 0 ,$ 
$$ A V (x,y) \le - \kappa V(x,y) + L_{+-} (x,y) , $$
which allows to conclude as before.

{\bf Part 2.2} The cases $ x \geq |c_{-+}|/N_-  , 0 \geq y >    - c_{+-}/N_+   $ or $ 0 \le x < |c_{-+}|/N_- , y \leq   - c_{+-}/N_+    $ are treated analogously to Part 1.2.

{\bf Part 3} Suppose now that $ x \le -  c_{++} /N_+  , y \geq   - c_{- -} /N_-  . $ Then 
$$ A V (x, y ) = A^1 V_{-+} (x, y) + A^2 V_{-+} (x,y ) = a_{-+} x^2 + b_{-+} xy + d_{-+} y^2 + L_{-+} (x,y ) ,$$
where $L_{-+} $ is a polynomial of degree $1$ and where
\begin{eqnarray*}
a_{-+} &=& -2 \nu_+ p    , \\
b_{-+} &=&  (c_{++}^* -  c_{--}^* ) ( \nu_+ + \nu_- - c_{--} ) + 2 p c_{-+} \\
d_{-+} & =& - c_{-+} (c_{++}^* + c_{--}^* ) .
\end{eqnarray*}  
Notice that by choice of $p,$ $ b_{-+} > 0 .$ The conclusion of this part follows analogously to the previous parts 1.1 and 2.1. 

{\bf Part 4} Suppose finally that $ x \le -  c_{++}/N_+  , y \le    - c_{+ -} /N_+   . $ Then 
$$ A V (x, y ) = A^1 V_{--} (x, y) + A^2 V_{--} (x,y ) = a_{--} x^2 + b_{--} xy + d_{--} y^2 + L_{--} (x,y ) ,$$
where $L_{--} $ is a polynomial of degree $1$ and where
\begin{eqnarray*}
a_{--} &=& -2 \nu_+ p    , \\
b_{--} &=&  (c_{++}^* -  c_{--}^* ) (\nu_+ + \nu_- ) \\
d_{--} & =& - 2 \nu_- q  ,
\end{eqnarray*}  
leading to the same conclusion as in the previous parts. 
\end{proof}

As a consequence of Proposition \ref{prop:lyapunov}, the process $X$ is stable in the sense that it necessarily possesses invariant probability measures, maybe several of them. The uniqueness of the invariant probability measure together with the Harris recurrence will follow from the following local Doeblin type lower bound.

\begin{prop}\label{thm:Doeblin}
For all $ T> 0 $ and for all $z_* = (x_*, y_*)  \in  \R^2 $ the following holds. There exist $R > 0 , $ an open set $ I \subset \R^2 $ with strictly positive Lebesgue measure and a constant $\beta \in (0, 1), $ depending on $I , R$ and the coefficients of the system with  
\begin{equation}\label{doblinminorization}
P_{T} (z , dz' ) \geq \beta 1_C  (z) \nu ( dz') ,
\end{equation} 
where $ C = B_R ( z_* ) $ is the (open) ball of radius $R$ centred at $z_* ,$ and where $ \nu $ is the uniform probability measure on $  I.$ 
\end{prop} 

\begin{proof}
We start with the case $ \nu_+ \neq \nu_-  ,$ under the assumption that $ c_{++}, c_{--}, c_{+-}, c_{-+} \neq 0 . $ In this case,  \cite{Clinet} in the proof of their Lemma 6.4 establish the lower bound \eqref{doblinminorization} for the four-dimensional Markov process 
$ \bar X = (X_{++}, X_{+-}, X_{-+}, X_{--} ) $ given by   
$$
 X_{++}  (t) =  e^{- \nu_+ t } X_{++} (0)  + \frac{c_{++}}{N_+}  \sum_{i=1}^{N_{+}}\int_0^t e^{- \nu_+ ( t- s) } Z_{+}^i (d s) , $$
$$ X_{-+} (t)  = e^{- \nu_+ t } X_{-+} (0)+ \frac{c_{-+}}{N_-}  \sum_{j=1}^{N_{-}}\int_0^t e^{- \nu_+ ( t- s) } Z_{-}^j  (ds),$$
$$X_{+-} (t) =  e^{- \nu_- t } X_{+-} (0) + \frac{c_{+-}}{N_+} \sum_{i=1}^{N_{+}}\int_0^t e^{- \nu_- ( t- s) }  Z_{+}^i (ds) ,$$
$$ X_{--} (t) =  e^{- \nu_- t } X_{--} (0)+   \frac{ c_{--}}{N_-}  \sum_{j=1}^{N_{-}}\int_0^t e^{- \nu_- ( t- s) }  Z_{-}^j  (ds) ,
$$ 
where $ X_{++} (0) + X_{-+} (0)  = X_+ (0)   , X_{--} (0)  + X_{+-} (0)  = X_- (0)  .$ 

More precisely, they show that for any $ \bar z_* \in \R^4 ,$ there exist $\bar R > 0 , $ an open rectangle $ \bar I \subset \R^4 $ with strictly positive Lebesgue measure and a constant $\bar \beta \in (0, 1), $ such that 
$$\bar P_{T} (\bar z , d\bar z' ) \geq \bar \beta 1_{\bar C} (\bar z) \bar \nu ( d \bar z') ,
$$
where $ \bar C = B_R ( \bar z_* ) $ is the (open) ball of radius $\bar R$ centred at $\bar z_* ,$ and where $ \bar \nu $ is the uniform probability measure on $ \bar  I.$ The above formula can be interpreted in the following way:  For any $ \bar z \in \bar C, $ with probability $\bar \beta, $ the law of $ \bar X (T) $ is equal to the law of $ U = (U_1, U_2, U_3, U_4) $ where $ U $ is a  uniform random vector on $ \bar I .$ Since $ \bar I $ is supposed to be a rectangle, this implies in particular the independence of its coordinates $ U_1, \ldots , U_4.$ 

Notice that we have $ X (T) = A \bar X (T) ,$ where   
$$ A = \left( \begin{array}{cccc}
1&0&1&0\\
0&1&0&1
\end{array}
\right) .$$
We now show how the above result implies the local lower bound for the original process $X.$ For that sake
let  $ z_* \in \R^2 $ be arbitrary and fix any $ \bar z_* \in \R^4 $ such that $A \bar z_* = z_* .$ Let $\bar R$ be the associated radius and choose $R$ such that $ B_R ( z_* ) \subset A B_{\bar R} ( \bar z_*) .$ Then for all $ z \in B_R (z_* )  $ and $ \bar z \in B_{\bar R} ( z_* ) $ with $ A \bar z = z, $  
$$ P_z ( X (T) \in \cdot ) =  P_{\bar z} (A  \bar X (T)  \in \cdot ) \geq \bar \beta \P ( A U \in \cdot )  .$$
Since 
$$ A U = \left( \begin{array}{c}
U_1 + U_3 \\
U_2 + U_4
\end{array}
\right) ,$$
by independence of the coordinates $ U_1, \ldots , U_4,$ this implies the desired result for the two-dimensional Markov process $X $ as well.

We finally deal with the case $ \nu_+ = \nu_- $ and $( c_{++},c_{+-} ) ,( c_{-+},c_{--}) $ linearly independent.  Fix $ z_* = (x_*, y_* ) $ and $ M> |x_*|+ |y_*|  $ arbitrarily and let $ H := \{ z = (x, y ) : |x| \le M, |y|\le M \} .$ Recall \eqref{eq:aplus} and  introduce finally the event  $E$ given by 
\begin{itemize}
\item $\pi^1_{+} ( [ 0,T] \times [  0,a^1_{+}]) =1,$
\item $\pi_{+}^1 ( [ 0,T] \times ]a^1_{+}, a^1_{+}+c_{++}  + M ) =0,$
\item $\pi_{+}^i ( [ 0,T] \times  [ 0, a^i_{+}+c_{++} + M  ) =0$ for all $ 2 \le i \le N_+,$ 
\item $\pi^1_{-} ( [ 0,T] \times [  0,a^1_{-}]) =1,$
\item $\pi_{-}^1 ( [ 0,T] \times  ] a^1_{-}, a^1_{-}+c_{+-}  + M ) =0,$
\item $\pi_{-}^j ( [ 0,T] \times  [ 0, a^j_{-}+c_{+-} + M  ) =0$ for all $ 2 \le j \le N_- .$ 
\end{itemize}
Define the substochastic kernel
$$
Q^T_{z} ( A) =P_z( E\cap \{ X ({T}) \in A\} )=P( E) P_z(  X ({T}) \in A   |   E) .
$$
The conditional law of $ X ({T}) $ given $  E,$ under $P_z,$ is equal to the law of
$$
Y_z ({T}) =ze^{-\nu_+  T}+e^{-\nu_+ U_{+}}\left(\begin{array}{c}c_{++}/N_+ \\c_{+-}/N_+ \end{array}\right)+ e^{-\nu_+  U_{-}}\left( \begin{array}{c}c_{-+}/N_-\\c_{--}/N_- .\end{array}\right) ,
$$
where the two jump-times $U_{+},U_{-}$ are independent uniform variables on $[ 0,T] .$ Since $$C=\left( \begin{array}{cc} c_{++}/N_+  & c_{-+}/N_-\\c_{+-}/N_+& c_{--}/N_- \end{array}
\right) $$ 
is invertible and the law of  $( e^{-\nu_+ U_{+}},e^{-\nu_+ U_{-}})$ is equivalent with the Lebesgue measure on $[ e^{-\nu_+ T},1] ^2, $ the law of $ Y_z ({T})$ has density
$$
f_{z}:v\mapsto | det\; C|^{-1} f\circ C^{-1}( v-ze^{-\nu_+ T}),
$$
where $f$ is the density of $( e^{-\nu_+ U_{+}}, e^{-\nu_+ U_{-}})$. The density is positive on the interior of its support
$$
supp( Y_z ({T}))=  e^{-\nu_+ T}z +C  [ e^{-\nu_+ T},1]^2 .
$$
Since $C$ is a homeomorphism, it is an open mapping. Thus we can find balls $B_r ( v_{0})\subset B_{2r} ( v_{0}) \subset C [ e^{- \nu_+ T},1]^2  $ for all $T>1.$ Take now $T$ so large that $e^{-\nu_+ T}\sup_{v \in  H} \|  v \| <r.$ For such $T$ and all $ z \in H$ we have
$$
  \overline{B}_r ( v_{0}) \subset e^{-\nu_+ T} z + B_{2r} ( v_{0})\subset supp( Y_z ({T})) .
$$
Note now that $H\times \overline{B}_r ( v_{0}) \ni (z,v) \mapsto f_{z}( v) $ is continuous, so the positivity of the density gives $\inf_{z \in H,v\in \overline{B}_r ( v_{0}) } f_{z}( v):=\alpha>0.$ We therefore conclude that
$$
Q^T_{z}( A)\geq P( E) \cdot \alpha \cdot  \lambda ( A\cap B_r ( v_{0}) ),
$$  
for all $z \in H,$ where $ \lambda $ denotes the Lebesgue measure on $\R^2.$ This proves the desired result.
\end{proof} 

We do now dispose of all ingredients to conclude the proof of Theorem \ref{theo:harris}.

\begin{proof}[Proof of Theorem \ref{theo:harris}]
1) We apply Proposition \ref{thm:Doeblin} with $z_* = 0 .$ Let $R$ be the associated radius. 

By Proposition \ref{prop:lyapunov}, we know that for a suitable compact set $K \subset \R^2,  $  $X $ comes back to $K $ infinitely often almost surely. For $ z = ( x, y ), $ write 
\begin{equation}\label{eq:flowy}
 \varphi_t (z) =  ( \varphi^{(1)}_t ( x) ,\varphi^{(2)}_t (y) )  = (e^{ - \nu_+  t}x , e^{- \nu_- t} y)
\end{equation}
for the flow of the process in between successiv jumps and let $ \| z\|_1 := |x| + |y|.$ Then
\begin{equation}
 \sup_{z \in K,  t \geq 0} \| \varphi_t (z) \|_1 := F  < \infty \; \; \mbox{ and } \; \;   \sup_{z \in K} \| \varphi_t (z) \|_1  \to 0 
\end{equation}
as $t \to \infty .$ Therefore there exists $t_* $ such that $\varphi_t (z) \in B_{R } ( 0) $ for all $t \geq t_* , $ for all $ z \in K .$ 
Hence, 
$$   \inf_{z\in K} P_z ( X ({t_* + s  })  \in B_R ( 0 ), 0 \le s \le 2T  )    > 0 . $$ 
Consequently, the Markov chain $(X ({kT}))_{k \in \N}  $ visits $ B_{R } (  0 )$ infinitely often almost surely. \\
The standard regeneration technique (see e.g. \cite{dashaeva}) allows to conclude that $(X ({kT}))_{k \in \N} $ and therefore $(X(t))_t $ are Harris recurrent. This concludes the proof of the Harris recurrence of the process.

2) The sampled chain $ (X ({kT }))_{k \geq 0 }$ is Feller according to Proposition \ref{prop:Feller}. Moreover it is $ \nu-$irreducible, where $ \nu $ is the measure introduced in Proposition \ref{thm:Doeblin}, associated with the point $z_* = (0,0) .$ Since $\nu$ is the uniform measure on some open set of strictly positive Lebesgue measure, the support of $\nu $ has non-empty interior. Theorem 3.4 of \cite{MT1992} implies that all compact sets are `petite' sets of the sampled chain. The Lyapunov condition established in Proposition \ref{prop:lyapunov} allows to apply Theorem 6.1 of \cite{MT1993} which implies the second assertion of the theorem.
\end{proof}

\end{document}